\documentclass[a4paper,11pt]{article} 
\usepackage[top=2.5cm, bottom=2.5cm, left=2.5cm, right=2.5cm]{geometry}
\usepackage{t1enc} 
\usepackage[utf8]{inputenc}
\usepackage{amssymb,amsmath,amsthm} 
\usepackage{graphicx,hyperref}
\usepackage{pgf,tikz}
\usepackage{hyperref}
\usetikzlibrary{shapes,backgrounds,arrows}
\usepackage{mathrsfs}
\usetikzlibrary{arrows}
\newtheorem{theorem}{Theorem}[section]
\newtheorem{proposition}[theorem]{Proposition}
\newtheorem{corollary}[theorem]{Corollary}
\newtheorem{lemma}[theorem]{Lemma}
\newtheorem{definition}[theorem]{Definition}
\newtheorem{remark}[theorem]{Remark}

\title{On edge-girth-regular graphs: lower bounds and new families}
\author{István Porupsánszki \footnote{This research was supported in part  by the Hungarian National Research, Development and Innovation Office  OTKA grant no. SNN 132625}\footnote{Department of Geometry and MTA-ELTE Geometric and Algebraic Combinatorics Research Group, Eötvös
Loránd University, 1117 Budapest, Pázmány s. 1/c, Hungary}}
\date{}

\begin{document}

\maketitle
\section*{Abstract}
An edge-girth-regular graph $egr(n,k,g,\lambda)$ is a $k-$regular graph of order $n$, girth $g$ and with the property that each of its edges is contained in exactly $\lambda$ distinct $g-$cycles. We present new families of edge-girth regular graphs arising from generalized quadrangles and pencils of elliptic quadrics.\\
An $egr(n, k, g, \lambda)$ is called extremal for the triple $(k, g, \lambda)$ if $n$ is the smallest order of any
$egr(n, k, g, \lambda)$. We give new lower bounds for the order of extremal edge-girth-regular graphs using properties of the eigenvalues of the adjacency matrix of a graph.
\section{Introduction}\label{one}
        This paper deals with simple, finite graphs, i.e., undirected graphs with no loops and no multiple edges.A graph is $k-$regular if every vertex has exactly $k$ distinct neighbours. It is of girth $g$ if its smallest cycle consists of $g$ edges. A $g-$cycle or girth cycle is a cycle of length $g$. The number of vertices is called the order of the graph.\\
In extremal graph theory, one often considers a problem with the following type: we fix some graph parameters or some graph properties and want to deduce the extremal number of another parameter (in many cases, the number of vertices or edges). The problem considered in our paper is motivated by the cage problem. Jajcay, Kiss, and Miklavič \cite{origin} defined a new type of regularity.
\begin{definition}
        An edge-girth-regular graph is a k-regular graph of order $n$, girth $g$, and with the property that each of its edges is contained in exactly $\lambda$ distinct $g-$cycles. They are denoted by $egr(n, k, g, \lambda)$.\\
        An $egr(n, k, g, \lambda)$ is called extremal for the triple $(k, g, \lambda)$ if $n$ is the smallest order of any $egr(n, k, g, \lambda)$. We denote the order of this extremal graph with $n(k,g,\lambda)$. If $G$ is an extremal bipartite $egr(n,k,g,\lambda)$, then we denote its order with $n_2(k,g,\lambda)$.
    \end{definition}
     The widely studied cage problem consists of finding $k-$regular graph of girth $g$ and minimal order, which we call a $(k,g)-$cage. For $k=2$ the $(k,g)-$cages are the $g-$cycles so without loss of generality, we assume that $k>2$ throughout this paper.
    \begin{theorem}[Moore bound \cite{moore}]
    Let $G$ be a $k-$regular graph of girth $g$ with $k>2$. Then the order of $G$ is at least $n_0(k,g)$, where
    $$n_0(k,g)= 
\begin{cases}
\frac{k(k-1)^{\frac{g-1}{2}}-2}{k-2}, \quad\text{if g is odd;}\\
\frac{2(k-1)^\frac{g}{2}-2}{k-2}, \quad\text{~~if g is even.}
\end{cases}$$
    \end{theorem}
    If a $G$ graph is a $(k,g)-$cage of order $n_0(k,g)$, then it is called a Moore cage.
        The motivation behind the definition of edge-girth-regularity is the fact that all of the Moore cages are also edge-girth-regular graphs.\\
        Recently, Drglin, Filipovski, Jajcay, and Raiman in \cite{bound} improved the lower bounds of the order of edge-girth-regular graphs.
    \begin{theorem}[Drglin, Filipovski, Jajcay, Raiman]{\label{lowerbound}}
Let $k$ and $g$ be a fixed pair of integers greater than or equal to $3$, and let $\lambda\le(k-1)^\frac{g-1}{2}$, if $g$ is odd and  $\lambda\le(k-1)^\frac{g}{2}$ if $g$ is even. Then
$$n(k,g,\lambda)\ge n_0(k,g)+ 
\begin{cases}
(k-1)^\frac{g-1}{2}-\lambda, \quad\text{if g is odd;}\\
\left\lceil2\frac{(k-1)^\frac{g}{2}-\lambda}{k}\right\rceil, \quad\text{~~if g is even.}
\end{cases}$$
Moreover, 
$$n_2(k,g,\lambda)\ge n_0(k,g)+2\left\lceil\frac{(k-1)^\frac{g}{2}-\lambda}{k}\right\rceil.$$
\end{theorem}
\begin{remark}
This theorem gives us a $\Theta\left(k^{\left[\frac{g}{2}\right]-1}\right)$ lower bound for the order of extremal $egr(n,k,g,\lambda)$ graphs.
\end{remark}
        The paper is organized as follows. In Section \ref{two}, we introduce new families of edge-girth-regular graphs arising from generalized quadrangles and pencils of elliptic quadrics. In Section \ref{four}, we prove that the second family of edge-girth-regular graphs constructed in Section \ref{two} is extremal. Here we use the properties of the eigenvalues of the adjacency matrix of a graph to give new lower bounds for the order of extremal edge-girth regular graphs of even girth. In Section \ref{five}, with a useful lemma, we extend the lower bounds to edge-girth-regular graphs of odd girth.
\section{New families of edge-girth regular graphs}\label{two}
This section introduces two infinite families of bipartite edge-girth-regular graphs arising from generalized quadrangles and pencils of elliptic quadrics. For a detailed introduction to generalized polygons, ovoids, spreads and quadrics, we refer the reader to the books \cite{ksz} and \cite{payne}.
\begin{definition}
Let $\Pi$ be a finite projective plane of order $q$. A biaffine plane is obtained from $\Pi$ by choosing a point-line pair $(P,l)$ and deleting $P$, $\ell$, all the lines incident with $P$ and all the points belonging to $\ell$. If the point-line pair is incident in $\Pi$, then we call the biaffine plane type 1, otherwise, type 2.
\end{definition}
It was proved by Araujo-Pardo and Leemans \cite{araujo} that the incidence graph of a biaffine plane of type 1 is an extremal bipartite edge-girth-regular graph.
\begin{theorem}[Araujo-Pardo, Leemans]
The incidence graph $\mathcal{B}_q$ of the biaffine plane of order $q$ of type 1 is an extremal $egr(2q^2,q,6,(q-1)^2(q-2))$ graph.
\end{theorem}
It was proved by Kiss, Miklavič, and Szőnyi \cite{kiss} that the other type of biaffine plane also gives us a new family of edge-girth-regular. We show that this graph is also extremal.
\begin{theorem}
The incidence graph $\mathcal{B}_q$ of a biaffine plane of order $q$ of type 2 is an extremal $egr(2q^2-2,q,6,(q-1)(q^2-3q+3))$ graph.
\end{theorem}
\begin{proof}
The lower bound in Theorem \ref{lowerbound} gives that 
$$n_2(q,6,(q-1)(q^2-3q+3))\ge 2(q^2-q+1)+2\left\lceil\frac{(q-1)^3-(q-1)(q^2-3q+3))}{q}\right\rceil=$$
$$=2(q^2-q+1)+2(q-2)=2q^2-2.$$
The order of $\mathcal{B}_q$ is exactly $2q^2-2$. Hence it is an extremal bipartite edge-girth-regular graph. 
\end{proof}
The previous results can be extended to generalized quadrangles in the natural way. However, only the case of type $1$  gives us a new family of edge-girth-regular graphs.
\begin{lemma}\label{const}
Let $\mathcal{Q}_q$ be a generalized quadrangle of order $(q,q)$, $P$ a point and $\{e_0,e_1,\ldots,e_q\}$ the set of lines incident with $P$. Consider the incidence graph $\mathcal{C}_q$ that is obtained from the generalized quadrangle by deleting $P$, all the points belonging to the set of lines $\{e_0,e_1,\ldots,e_q\}$, and all the lines that are incident with a point belonging to $e_0$. Then $\mathcal{Q}_q$ is a $q-$regular graph of order $2q^3$ and girth $8$.
\end{lemma}
\begin{proof}
By definition, $\mathcal{C}_q$ is obtained by deleting $1+(q+1)q$ points and $q+1+q^2$ lines, hence $\mathcal{C}_q$ has $2(q^3+q^2+q+1)-2(q^2+q+1)=2q^3$ vertices. 
We show now that $\mathcal{C}_q$ is $q-$regular. Note that for every non-incident point-line pair $(P,l)\in\mathcal{Q}_q$ there exists a unique point-line pair $(P',l')$ such that $P\mathrm{I}l'\mathrm{I}P'\mathrm{I}l$, where $\mathrm{I}$ denotes the incidence relation. We can use this fact for regularity. Every deleted line has a point incident with $e_0$, hence every point has exactly one deleted line because of the previous fact. On the other hand, every deleted point has a line incident with $P$, hence every point has exactly one deleted line. These induce that $\mathcal{C}_q$ is $q-$regular because the incidence graph of $\mathcal{Q}_q$ is $(q+1)-$regular. The $\mathcal{C}_q$ graph is contained in the incidence graph of $\mathcal{Q}_q$, hence the length of the minimal cycle of $\mathcal{C}_q$ is at least $8$.\\
\end{proof}
\begin{theorem}
Suppose a generalized quadrangle of order $(q,q)$ exists. Then there exists an $egr(2q^3,q,8,(q-1)^2((q-2)^2+1)$ graph.
\end{theorem}
\begin{proof}
We show that every edge of the graph $\mathcal{C}_q$ constructed in the previous lemma is contained in exactly $(q-1)^2((q-2)^2+1)$ $8-$cycles. Let $C=(u_1,u_2,u_3,u_4,u_5,u_6,u_7,u_8)$ be a $8-$cycle. We count how many unchanged $8-$cycles contain the edge $(u_1,u_2)$. Without loss of generality, we may assume that $u_1$ is a line and $u_2$ is a point. Then we  have $(q-1)$ possible choices for the  neighbours $u_3\ne u_1$ and $u_4\ne u_2$. 
\begin{center}

\includegraphics[scale=0.5]{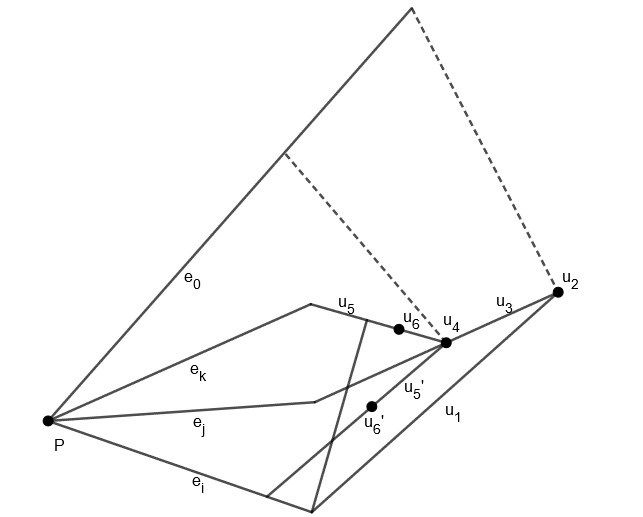} 

\end{center}
There are two types of $u_5$: one intersects the same line $e_i$ that is connected with $u_1$ and the other $q-1$ possible choices. Let us denote that unique line by $u_5'$ and any other neighbour of $u_4$ by $u_5$. The intersection of $e_i$ and $u_0$ can be connected uniquely with $u_5$, hence we have $q-2$ choices for $u_6$, but $u_6$ and $u_1$ uniquely define $u_7$ and $u_8$. Similarly, we have $q-1$ possible choices for the neighbour of $u_5'$ and $u_6'$ and $u_1$ uniquely define $u_7$ and $u_8$. So far, we have $(q-1)^2((q-2)^2+q-1)$ different $8-$cycles that contain the edge $(u_1,u_2)$. Unfortunately, the line $u_7$ is sometimes deleted. We have to count these cases and subtract the result to get the exact value.
\begin{center}

\includegraphics[scale=0.5]{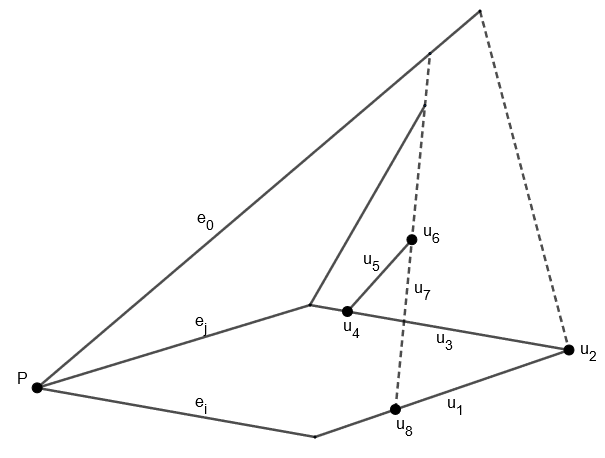} 

\end{center}
We have $(q-1)$ possible choices for a neighbour $u_3\ne u_1$ and a neighbour $u_8\ne u_2$. There is a unique deleted line through $u_8$, we denote it by $u_7$. The intersection of $e_j$ and $u_3$ can be connected uniquely with $u_7$, hence we have $q-2$ choices for $u_6$. Now, $u_4$ and $u_5$ are uniquely defined by $u_3$ and $u_6$. So we have $(q-1)^2(q-2)$ distinct $8-$cycles that contain the edge $(u_1,u_2)$ when $u_7$ is deleted. Therefore the number of $8-$cycles that contain a certain edge in $C_q$ is
$$(q-1)^2((q-2)^2+q-1)-(q-1)^2(q-2)=(q-1)^2((q-2)^2+1).$$
\end{proof}
\begin{corollary}
If $q$ is a prime power, then there exists a generalized quadrangle of order $(q,q)$ therefore there exists an $egr(2q^3,q,8,(q-1)^2((q-2)^2+1)$ graph.
\end{corollary}
\begin{remark}
Theorem \ref{lowerbound} gives the lower bound
$$n_2\left(q,8,(q-1)^2((q-2)^2+1))\right)\ge 2(q^3-6q+10).$$
The excess is $2(6q-10).$ We can get a slightly better result in a very special case.
\end{remark}
\begin{theorem}[Kiss, Miklavič, Szőnyi \cite{kiss}]
Let $\mathcal{G}=(\mathcal{P},\mathcal{L},\mathrm{I})$ be a generalized quadrangle of order $q$ which admits both an ovoid $\mathcal{O}$ (a set of $q^2 +1$ points, no two of which are collinear) and a spread $\mathcal{S}$ (a set of $q^2+1$ line, no two of which intersect). Delete the points of $\mathcal{O}$ and the lines of  $\mathcal{S}$. Then the Levi graph of $(\mathcal{P}\setminus\mathcal{O},\mathcal{L}\setminus\mathcal{S},\mathrm{I})$ is an $egr(2q(q^2+1),q,8,(q-1)^2(q-2)^2).$
\end{theorem}
\begin{remark}
Theorem \ref{lowerbound} gives the lower bound
$$n_2\left(q,8,(q-1)^2(q-2)^2)\right)\ge 2(q^3-5q+8).$$
The excess is $2(6q-8),$ and it is slightly better than the previous construction. However, the only known generalized quadrangles admitting both an ovoid and a spread are the generalized quadrangles $W(q)$ for $q$ even. The construction in Lemma \ref{const} works for any generalized quadrangle of order $(q,q)$.
\end{remark}
Our next construction is based on the following theorem whose proof can be found for example in \cite{Hirschfeld} or \cite{PI}. This is an extremal bipartite edge-girth-regular graph and is special in that its parameter $\lambda$ is relatively small compared to the previously seen graphs.
\begin{theorem}
The points of $PG(3,q)$ can be covered by the union of $q+1$ pairwise disjoint elliptic quadrics. The elements of the covering form a pencil of quadrics.
\end{theorem}
An important property of elliptic quadrics is that they are also ovoids, so they do not have three collinear points.
\begin{theorem}{\label{new}}
Let $q$ be a prime power. Then there exists a bipartite  $egr(2(q^3+q^2+q+1),q^2+q+1,4,q^3+q^2)$ graph.
\end{theorem}
\begin{proof}
Let make two copies of the points of $PG(3,q)$ and denote the sets with $\mathcal{P}$ and $\mathcal{P}$', respectively. Similarly, $p\in\mathcal{P}$ and $p'\in\mathcal{P}'$ represent the same point. \\
By the previous theorem, we construct the edges of the graph in the following way. Every point $P$ is contained in exactly one ovoid. In this ovoid, there is a unique tangent plane $S_p$ through $p$. Let $(p,r')\in E$ if and only if $r\in S_p$. 
Now we show that this graph is a bipartite edge-girth-regular graph. \\
First, we prove that every plane is a tangent plane for exactly one point, therefore it is a bijection. Let $p$ be a point and $S_p$ its tangent plane. The other $q$ ovoids intersect $S_p$ in at most $(q+1)$ points since an oval of a plane contains $(q+1)$ points But if we use this upper bound, we get
$$1+q(q+1)=q^2+q+1.$$
It is the number of points of $S_p$, therefore it is indeed a bijection between the points and planes.
Now we focus on determining the parameters of the graph. It is bipartite and has a valency of $(q^2+q+1).$ (For every point $p$ $(p,p')\in E$.) \\
Since it is a bipartite graph, its girth must be even so now it is enough to show that every edge is contained in exactly $q^3+q^2$ distinct $4-$cycles. Suppose that $p_1p_2'\in E$, so equivalently $p_2\in S_{p_1}$. We finish the $4-$cycle $(p_1,p_2',p_3,p_4')$ and count the number of possible ways it can be done. For $p_4'\in S_{p_1}$ we $|S_p|-1=q^2+q$ possible choices since $p_2'$ is already taken. $p_2$ and $p_4$ is contained in exactly $q+1$ planes. Each plane is a tangent plane for a unique point, therefore for $p_3$ we have $q+1-1=q$ possible choices. Now we have that each edge is contained in exactly $(q^2+q)q=q^3+q^2$ distinct $4-$cycles.
\end{proof}
\begin{remark}
Theorem \ref{lowerbound} gives the lower bound 
$$n_2(q^2+q+1,4,q^3+q^2)\ge 4q^2+2q+2.$$
This is much lower than the orders of the previously constructed graphs. In the next section, we prove that this graph is an extremal bipartite edge-girth regular graph.
\end{remark}
\section{A new lower bound for the order of edge-girth regular graphs of even girth}\label{four}
We recall the most important properties of eigenvalues of adjacency matrices. The proofs of the following theorems can be found in \cite{eigenvalues}.
\begin{definition}
The adjacency matrix $A(G)$ of a simple graph $G=(V,E)$ is defined as follows: it is a symmetric matrix of size $|V|\times|V|$ labeled by the vertices of $G$, and 
$$A(G)_{u,v}=
\begin{cases}
1, ~if~ (u,v)\in E(G)\\
0, ~if~ (u,v)\not\in E(G)
\end{cases}$$
\end{definition}
\begin{proposition}
If $G$ is a simple graph, then the eigenvalues of its adjacency matrix $A$ satisfies $\sum \lambda_i = 0$ and $\sum \lambda_i^2= 2e(G)$, where $e(G)$ denotes the number of edges of G. \\In general, $\sum\lambda_i\ell$ counts the number of closed walks of length $\ell$.
\end{proposition}
Let us now focus on the largest and the smallest eigenvalues denoted by $\lambda_1$ and $\lambda_n$, respectively.
\begin{theorem}
(a) We have $|\lambda_n| \le \lambda_1$.\\
(b) Let $G$ be a connected graph and assume that $-\lambda_n$ = $\lambda_1$. Then $G$ is bipartite.\\
(c) $G$ is a bipartite graph if and only if its spectrum is symmetric to $0$.
\end{theorem}
\begin{theorem}\label{regular}
Let $G$ be a $k$-regular graph. Then $\lambda_1 = k$ and its multiplicity is the
number of components. Every eigenvector belonging to $k$ is constant on each component. In particular, the multiplicity of $\lambda_1$ is $1$ for connected graphs.
\end{theorem}
First, we consider the case $g=4$. Later on, we generalize it for even girth.
\begin{theorem}\label{egr4}
Let $G$ be an extremal $egr(n,k,4,\lambda)$ graph. Then
$$n=n(k,4,\lambda)\ge\frac{k^3-2k^2+2k-1+\lambda}{\lambda+k-1}.$$
If $G$ is an extremal bipartite $egr(n,k,4,\lambda)$ graph, then 
$$n=n_2(k,4,\lambda)\ge\frac{2(k^3-2k^2+2k-1+\lambda)}{\lambda+k-1}.$$
\end{theorem}
\begin{proof}
By Theorem \ref{regular}, $\lambda_1=k$ because $G$ is $k-$regular. The number of closed walks of length $\ell$ equals $\sum\lambda_i^\ell$. The number of closed walks of length two equals to twice the number of edges. If we count the number of incident (point, edge) pairs in two ways, we get $2|e(G)|=nk.$ Now we have
$$\sum_{i=1}^n\lambda_i^2=2|e(G)|=nk.$$
Similarly, we calculate the number of closed walks of length four. We have two cases depending on whether the closed walk is a circle or not. The number of $4$-cycles is $nk\lambda$, because we can start at any vertex, every vertex is incident with $k$ edges and every edge is contained in exactly $\lambda$ distinct $4-$cycles. If a closed walk of length four does not contain cycles, then its undirected edges form a tree with two edges. Choose any vertex for the first point of the closed walk. We have $k$ possible choices for the second vertex. If the third vertex is the first one, then we again have $k$ possible choices for the fourth vertex. If the first and third vertices are different, then the fourth vertex must be the same as the second one, thus the number of closed walks of length four
$$\sum_{i=1}^n\lambda_i^4=nk\lambda+nk^2+nk(k-1)=nk(\lambda+2k-1).$$
$G$ is $k-$regular therefore $\lambda_1=k$. Now we have the following equalities:
$$\sum_{i=2}^n\lambda_i^2=nk-k^2,$$
$$\sum_{i=2}^n\lambda_i^4=nk(\lambda+2k-1)-k^4.$$
Using the inequality between the arithmetic and quadratic means for the set $\{\lambda_2^2\ldots\lambda_n^2\}$ and some equivalent operations, we get the first inequality:
$$\left(\frac{\sum_{i=2}^n\lambda_i^2}{n-1}\right)^2\le\frac{\sum_{i=2}^n\lambda_i^4}{n-1},$$
$$\left(\frac{nk-k^2}{n-1}\right)^2\le \frac{nk(\lambda+2k-1)-k^4}{n-1},$$
$$(nk-k^2)^2\le (n-1)(nk(\lambda+2k-1)-k^4),$$
$$n^2k-2k^2n+k^3\le n^2(\lambda+2k-1)-n(\lambda+2k-1)-nk^3+k^3,$$
$$n(k^3-2k^2+2k-1+\lambda)\le n^2(\lambda+k-1),$$
$$\frac{k^3-2k^2+2k-1+\lambda}{\lambda+k-1}\le n.$$
If $G$ is $k-$regular and bipartite, then $\lambda_1=k=-\lambda_n$, hence we have the following equalities: 
$$\sum_{i=2}^{n-1}\lambda_i^2=nk-2k^2,$$
$$\sum_{i=2}^{n-1}\lambda_i^4=nk(\lambda+2k-1)-2k^4.$$
In the same way, using the inequality between the arithmetic and quadratic means for the set $\{\lambda_2^2\ldots\lambda_{n-1}^2\}$ and some equivalent operations, we get the second inequality:
$$\frac{2(k^3-2k^2+2k-1+\lambda)}{\lambda+k-1}\le n.$$
\end{proof}
\begin{corollary}
The complete bipartite graph $K_{k,k}$ is an extremal bipartite $egr(2k,k,4,(k-1)^2)$ graph.
\end{corollary}
\begin{proof}
The only property we have to check is that $K_{k,k}$ extremal. Theorem \ref{egr4} gives the lower bound
$$n_2(k,4,(k-1)^2)\ge2\frac{k^3-2k^2+2k-1+(k-1)^2}{(k-1)^2+k-1}=2\frac{k^3-k^2}{k^2-k}=2k.$$
\end{proof}
\begin{theorem}
The graph constructed in Theorem \ref{new} is an extremal bipartite\\
$egr(2(q^3+q^2+q+1),q^2+q+1,4,q^3+q^2)$ graph.
\end{theorem}
\begin{proof}
The previous inequality gives a lower bound for $n_2(q^2+q+1,4,q^3+q^2)$ and it is exactly $2(q^3+q^2+q+1).$
\end{proof}
\begin{theorem}
áIf a bipartite $egr(n,k,4,\lambda)$ graph satisfies the inequality of Theorem \ref{egr4} with equality, then its spectrum is
$$\left\{-k^{(1)};-\sqrt{\frac{nk-2k^2}{n-2}}^{\left(\frac{n-2}{2}\right)};\sqrt{\frac{nk-2k^2}{n-2}}^{\left(\frac{n-2}{2}\right)};k^{(1)}\right\}.$$
\end{theorem}
\begin{proof}
The proof of Theorem \ref{egr4} is based on the inequality between the arithmetic and quadratic means. Thus equality holds if and only if $\lambda_2^2=\ldots=\lambda_{n-1}^2$. We know the sum of these numbers:
$$\sum_{i=2}^{n-1}\lambda_i^2=nk-2k^2.$$
Hence
$$(n-2)\lambda_2^2=nk-2k^2,$$
$$\lambda_2^2=\frac{nk-2k^2}{n-2}.$$
The spectrum of a bipartite graph is symmetric to $0$, therefore if $G$ is an extremal bipartite $egr(n,k,4,\lambda)$ graph and its order satisfies the inequality in Theorem \ref{egr4}, then the spectrum of $G$ is
$$\left\{-k^{(1)};-\sqrt{\frac{nk-2k^2}{n-2}}^{\left(\frac{n-2}{2}\right)};\sqrt{\frac{nk-2k^2}{n-2}}^{\left(\frac{n-2}{2}\right)};k^{(1)}\right\}.$$
\end{proof}
\begin{corollary}
The spectrum of the graph constructed in Theorem \ref{new} is
$$\left\{-(q^2+q+1)^{(1)};-q^{\left(q^3+q^2+q\right)};q^{\left(q^3+q^2+q\right)};(q^2+q+1)^{(1)}\right\}.$$
\end{corollary}
\begin{remark}
This corollary also shows the strength of the inequality: if the spectrum of a bipartite edge-girth-regular graph of girth $4$ consists of four distinct eigenvalues, then the graph is an extremal bipartite edge-girth-regular graph.
\end{remark}
Theorem \ref{egr4} can be generalized for higher even girths, however, the results and the formulas are less elegant.
\begin{definition}
Let $G$ be an $egr(n,k,g,\lambda)$ graph. Let $c(\ell,k)$ denote the number of closed walks of length $\ell$ starting of a vertex $v$ that do not contain circles.
\end{definition}
\begin{proposition}
Let $\ell=2s$ be an even number. If $\ell\le g+1$, then $c(\ell,k)$ is a polynomial of $k$ of degree $\frac{\ell}{2}$ and its value does not depend on the choice of the vertex $v$. In this case the main coefficient of $c(2s,k)$ is the $s$-th Catalan number 
$$C_{s}=\binom{2s}{s}-\binom{2s}{s+1}.$$
Moreover, 
$$C_{s}k(k-1)^{s-1}\le c(2s,k)\le C_s k^s.$$
\end{proposition}
\begin{proof}
If we consider the fact that at each step, we either increase the distance from $v$ by one or decrease it by one, then we get that in such a closed walk, we increased this distance $s-$times. If we decrease the distance, then the walk goes through the previous edge. If we increase the distance, we have $k$ possible choices when we are at the starting vertex, and $k-1$ possible choices otherwise. If the main coefficient is $C$, then we proved that $c(2s,k)$ is indeed a polynomial of $k$ of degree $s$ and
$$Ck(k-1)^{s-1}\le c(2s,k)\le Ck^s.$$
Now, we only need to show that the main coefficient is the $s-$th Catalan number. We prove it with a combinatorial interpretation of the Catalan numbers. It is a well-known folklore result that $C_s$ is the number of monotonic lattice paths along the edges of a grid with $s \times s$ square cells, which do not pass above the diagonal. A monotonic path starts in the lower left corner, finishes in the upper right corner, and consists entirely of edges pointing rightwards or upwards. Counting such paths is equivalent to counting the type of closed walks of length $2s$: increasing the distance stands for "move right" and decreasing the distance stands for "move up".
\end{proof}
\begin{remark}
We present the list of the first few $c$ polynomials:
$$
\begin{aligned}
c(2,k)=&k,\\
c(4,k)=&2k^2-k,\\
c(6,k)=&5k^3-6k^2+2k,\\
c(8,k)=&14k^4-28k^3+20k^2-5k.
\end{aligned}
$$
\end{remark}
Now, we ready to prove the main result of our paper.
\begin{theorem}
Let $G$ be an $egr(n,k,g,\lambda)$ graph, where $g$ is even.\\
If $g\equiv0$ (mod $4$), then
$$n(k,g,\lambda)\ge\frac{c(g,k)+k\lambda+k^{g}-2c(\frac{g}{2},k)k^{\frac{g}{2}}}{c(g,k)-c^2(\frac{g}{2},k)+k\lambda},$$
$$n_2(k,g,\lambda)\ge2\frac{c(g,k)+k\lambda+k^{g}-2c(\frac{g}{2},k)k^{\frac{g}{2}}}{c(g,k)-c^2(\frac{g}{2},k)+k\lambda}.$$
If $g\equiv2$ (mod $4$), then
$$n(k,g,\lambda)\ge\frac{c(g,k)+k\lambda+k^g}{c(g,k)+k\lambda},$$
$$n_2(k,g,\lambda)\ge\frac{2k^g}{c(g,k)+k\lambda}.$$
\end{theorem}
\begin{proof}
$G$ is $k-$regular therefore $\lambda_1=k$. If $G$ is bipartite, then $\lambda_n=-k$. Let $g=2s$. Similarly to the case of girth four, we use the inequality between the arithmetic and quadratic means for the set $\{\lambda_2^s\ldots \lambda_n^s\}$:
$$\left(\frac{nc(s,k)-k^s}{n-1}\right)^2=\left(\frac{\sum_{i=2}^n\lambda_i^s}{n-1}\right)^2\le\frac{\sum_{i=2}^n\lambda_i^{2s}}{n-1}=\frac{nc(2s,k)+nk\lambda-k^{2s}}{n-1}.$$
First, if $s$ is even, then we have the following inequality
$$(nc(s,k)-k^s)^2\le (n-1)(nc(2s,k)+nk\lambda-k^{2s}).$$
After some calculations, we can give a lower bound for the order of the graph:
$$n^2c^2(s,k)-2nc(s,k)k^s+k^{2s}\le n^2(c(2s,k)+k\lambda)-n(c(2s,k)+k\lambda+k^{2s})+k^{2s},$$
$$n(c(2s,k)+k\lambda+k^{2s}-2c(s,k)k^s)\le n^2(c(2s,k)+k\lambda-c^2(s,k)),$$
$$\frac{c(2s,k)+k\lambda+k^{2s}-2c(s,k)k^s}{c(2s,k)+k\lambda-c^2(s,k)}\le n(k,g,\lambda).$$
If $s$ is odd we can get the same result but we also know that in this case $c(s,k)=c(\frac{g}{2})=0$.\\
Let $G$ be a bipartite edge-girth regular graph of girth $2s$. Then we  apply the inequality between the arithmetic and quadratic means for the set  $\{\lambda_2^s\ldots\lambda_{n-1}^s\}$:
$$\left(\frac{nc(s,k)-k^s-(-k)^s}{n-2}\right)^2=\left(\frac{\sum_{i=2}^{n-1}\lambda_i^s}{n-2}\right)^2\le\frac{\sum_{i=2}^{n-2}\lambda_i^{2s}}{n-2}=\frac{nc(2s,k)+nk\lambda-2k^{2s}}{n-2}.$$
Begin with the case when $s$ is even, then we have the following inequality
$$(nc(s,k)-2k^s)^2\le (n-2)(nc(2s,k)+nk\lambda-2k^{2s})$$
With elementary calculations we get a lower bound for $n_2(k,2s,\lambda)$:
$$n^2c^2(s,k)-4nc(s,k)k^s+4k^{2s}\le n^2(c(2s,k)+k\lambda)-2n(c(2s,k)+k\lambda+k^{2s})+4k^{2s},$$
$$2n(c(2s,k)+k\lambda+k^{2s}-2c(s,k)k^s)\le n^2(c(2s,k)-c^2(s,k)+k\lambda),$$
$$2\frac{c(2s,k)+k\lambda+k^{2s}-2c(s,k)k^s}{c(2s,k)-c^2(s,k)+k\lambda}\le n_2(k,g,\lambda).$$
If $s$ is odd, then $c(s,k)=0$ and $(-k)^s=-k^s$, hence left-hand side of the inequality is $0$. This means that 
$$0\le n(c(2s,k)+k\lambda)-2k^{2s},$$
$$\frac{2k^{2s}}{c(2s,k)+k\lambda}\le n_2(k,g,\lambda).$$
\end{proof}
\begin{remark}
Recall that Theorem \ref{lowerbound} gives a $\Theta\left(k^{\left[\frac{g}{2}\right]-1}\right)$ lower bound for the order of extremal $egr(n,k,g,\lambda)$ graphs. The previous inequalities gives $O\left(k^{\left[\frac{g}{2}\right]-1}\right)$ lower bounds for the order of extremal $egr(n,k,g,\lambda)$ graphs of even girth.\\
On the other hand, if $\lambda=o(k^{\frac{g}{2}-1})$, then the inequalities give $\Theta\left(k^{\left[\frac{g}{2}\right]}\right)$  lower bounds. It means that if  $\lambda$ is small, then we have a lower bound for the order of extremal edge-girth-regular graphs of even girth that is around $k-$times larger than the previously known lower bound. 
\end{remark}
\section{A useful lemma and an inequality for edge-girth-regular graphs of odd girth}\label{five}
We conclude our paper with an upper bound for the number of $(g+1)-$cycles through a certain vertex of an $egr(n,k,g,\lambda)$ graph of odd girth and a corollary of this result.\\
With another application of the Cauchy-Schwarcz inequality, we can give a new lower bound of the order of edge-girth regular graphs of odd girth. First, we need to prove the following lemma.

\begin{lemma}{\label{cycles}}
Let $G$ be an $egr(n,k,g,\lambda)$ graph, where $g=2h+1$ is an odd number. Then every vertex is contained in at most\ $\binom{k}{2}\left((k-1)^h-\frac{\lambda}{k-1}\right)$ distinct $(g+1)-$cycles.
\end{lemma}
\begin{proof}
Choose an arbitrary vertex $v$ and define the sets $D_i(v)$ of vertices as follows: $u\in D_i(v)$ if the length of the shortest $uv-$path is $i$.\\
First, count the number of edges in $D_h(v)$. Suppose that there is an edge $ab\in E$ in $D_h(v)$. Then $ab$, and the unique $h-$paths $av$ and $bv$ form a $g-$cycle. Since $G$ has girth $g$, this is a bijection between the edges in $D_h(v)$ and the $g-$cycles through $v$. So it is enough to count the number of $g-$cycles $v$. The graph is $k-$regular, hence $v$ lies on exactly $k$ edges and every edge is contained in exactly $\lambda$ distinct $g-$cycles, but in this way, every $g-$cycle through $v$ is counted twice. Therefore the number of $g-$cycles through $v$ is $\frac{k\lambda}{2}$, hence the number of edges in $D_h(v)$ is $\frac{k\lambda}{2}.$\\
Now count the number of edges between $D_h(v)$ and $D_{h+1}(v)$. There are $k(k-1)^{h-1}$ vertices in $D_h(v)$. These vertices have $k^2(k-1)^{h-1}$ edges but every edge in $D_h(v)$ is counted twice. The number of edges between $D_{h-1}(v)$ and $D_h(v)$ is $k(k-1)^{h-1}$, the number of edges in $D_h(v)$ is $\frac{k\lambda}{2}$, hence the number of edges between $D_h(v)$ and $D_{h+1}(v)$ is
$$k^2(k-1)^{h-1}-k(k-1)^{h-1}-2\frac{k\lambda}{2}=k\left((k-1)^h-\lambda\right).$$
\indent There are two types of $(g+1)-$cycles through $v$: the ones that reach $D_{h+1}(v)$ with two $(h+1)-$paths at the same vertex, and the ones that join two $h-$paths in $D_h(v)$ with a $2-$path. We give upper bounds for the number of $(g+1)-$cycles through $v.$\\
\indent If a vertex $w\in D_{h+1}(v)$ has $m$ neighbours in $D_h(v)$, then $v$ and $w$ can be with two distinct $(h+1)-$paths $\binom{m}{2}$ different ways, therefore the number of $(g+1)-$cycles through $v$ and $w$ is $\binom {m}{2}$. Thus there is a quadratic increase in the number of $(g+1)-$cycles through $v$ but $m\le k$ because the graph is $k-$regular. If we would like to maximize the number of $g+1-$cycles through $v$ that reaches $D_{h+1}(v)$, we have to choose some vertices in $D_{h+1}(v)$ that cover all the edges between $D_h(v)$ and $D_{h+1}(v)$ and have no other edges. Since the number of edges between $D_h(v)$ and $D_{h+1}(v)$ is $k\left((k-1)^h-\lambda\right)$, the number of $(g+1)-$cycles through $v$ that reach $D_{h+1}(v)$ is at most
$$\frac{k\left((k-1)^h-\lambda\right)}{k}\binom{k}{2}=\binom{k}{2}\left((k-1)^h-\lambda\right).$$
Similarly, if a vertex $w'\in D_h(v)$ has $m$ neighbours in $D_h(v)$, then $v$ and $w'$ can be connected with two $(h+1)-$paths in $\binom{m}{2}$ different ways. Therefore the number of $(g+1)-$cycles through $v$ and $w'$ is $\binom{m}{2}.$ Take the graph induced by the vertices of $D_h(v).$ Every degree is at most $k-1$ and the number of edges is $\frac{k\lambda}{2}$. Denote the set of degrees with $\{d_1,d_2,\ldots,d_{k(k-1)^h}\}.$ Now we give an upper bound for 
$$\sum_{i=1}^{k(k-1)^h}\binom{d_i}{2}=\sum_{i=1}^{k(k-1)^h}\frac{d_i^2-d_i}{2}=\frac12\sum_{i=1}^{k(k-1)^h}d_i^2-\frac{\lambda}{2}.$$
To do that we need to maximize the sum of squares such that the sum of degrees is constant, and $0\le d_i\le k-1$ for any $1\le i\le k(k-1)^h.$ Since $x^2+y^2\le (x+1)^2+(y-1)^2$ for any $x\le y$ real numbers, we obtain an upper bound if we set as many $d_i$ to $k-1$ as possible while the rest is $0$. Hence
$$\sum_{i=1}^{k(k-1)^h}d_i^2\le \frac{k\lambda}{k-1}(k-1)^2=\lambda k(k-1).$$
Therefore
$$\sum_{i=1}^{k(k-1)^h}\binom{d_i}{2}\le\frac{\lambda k(k-1)}{2}-\frac{\lambda k}{2}=\binom{k}{2}\lambda \left(1-\frac{1}{k-1}\right).$$
We gave upper bounds for each type of $(g+1)-$cycles through a certain vertex and now we add them to prove the lemma:
$$\binom{k}{2}\left((k-1)^h-\lambda\right)+\binom{k}{2}\lambda\left(1-\frac{1}{k-1}\right)=\binom{k}{2}\left((k-1)^h-\frac{\lambda}{k-1}\right).$$
\end{proof}
\begin{remark}
The upper bound of the previous lemma is sharp for the Petersen and Hoffmann-Singleton graphs.
\end{remark}
And now we conclude our paper with an application of Lemma \ref{cycles}.
\begin{proposition}{\label{IP_odd}}
Let $G$ be an $egr(n,k,g,\lambda)$ graph of odd girth. Then
$$n(k,g,\lambda)\ge\frac{k^{g+1}c(g-1,k)+k^{g-1}\left(c(g+1,k)+k(k-1)^{\frac{g+1}{2}}-\lambda k\right)-2k^{g+1}\lambda}{c(g-1,k)\left(c(g+1,k)+k(k-1)^{\frac{g+1}{2}}-\lambda k\right)-k^2\lambda^2}.$$
\end{proposition}
\begin{proof}
Let $k=\lambda_1\ge\ldots\ge\lambda_n$ be the eigenvalues of the adjacency matrix of $G$. Then
$$\sum_{i=1}^n\lambda_i^{g-1}=nc(g-1,k),$$
$$\sum_{i=1}^n\lambda_i^{g}=nk\lambda,$$
$$\sum_{i=1}^n\lambda_i^{g+1}=nc(g+1,k)+\sum_{v\in G}\textrm{(the number of directed (g+1)-cycles containing v)}.$$
Using the fact the the largest eigenvalue is $k$ and the upper bound for the number of $(g+1)-$cycles, we get
$$\sum_{i=2}^n\lambda_i^{g-1}=nc(g-1,k)-k^{g-1},$$
$$\sum_{i=2}^n\lambda_i^{g}=nk\lambda-k^g,$$
$$\sum_{i=2}^n\lambda_i^{g+1}=nc(g+1,k)+2n\binom{k}{2}\left((k-1)^h-\frac{\lambda}{k-1}\right)-k^{g+1}.$$
Let $g=2h+1.$ Then the Cauchy-Schwarz inequality gives
$$\left(\sum_{i=2}^n\lambda_i^g\right)^2=\left(\sum_{i=2}^n\lambda_i^{h}\lambda_i^{h+1}\right)^2\le \left(\sum_{i=2}^n\lambda_i^{2h}\right)\left(\sum_{i=2}^n\lambda_i^{2h+2}\right)=\left(\sum_{i=2}^n\lambda_i^{g-1}\right)\left(\sum_{i=2}^n\lambda_i^{g+1}\right).$$
Now we can give an upper bound for the right hand side and we get the following inequality:
$$\left(nk\lambda-k^g\right)^2\le \left(nc(g-1,k)-k^{g-1}\right)\left(nc(g+1,k)+2n\binom{k}{2}\left((k-1)^h-\frac{\lambda}{k-1}\right)-k^{g+1}\right).$$
We have a quadratic inequality in $n$. After simplifying the inequality we get the lower bound for the order of $G$:
$$n(k,g,\lambda)\ge\frac{k^{g+1}c(g-1,k)+k^{g-1}\left(c(g+1,k)+k(k-1)^{\frac{g+1}{2}}-\lambda k\right)-2k^{g+1}\lambda}{c(g-1,k)\left(c(g+1,k)+k(k-1)^{\frac{g+1}{2}}-\lambda k\right)-k^2\lambda^2}.$$
\end{proof}
\begin{corollary}
The Petersen graph is an extremal $egr(10,3,5,4)$ graph.
\end{corollary}
\begin{proof}
The theorem states that
$$n(k,5,\lambda)\ge \frac{3k^6+k^5-3k^4+k^3-2k^4\lambda-k^3\lambda}{2 k^4 + 3 k^3 - 8 k^2 + 5 k - 1-\lambda^2-2k\lambda+\lambda}.$$
Hence $n(3,5,4)\ge9.78$ so the order of the Petersen graph is extremal.
\end{proof}
\begin{remark}

Unfortunately, the lower bound presented in Proposition \ref{IP_odd} is weaker than the lower bound in Theorem \ref{lowerbound}.
However, the inequality in Lemma \ref{cycles} may be useful for further improvements for the lower bound of the order of edge-girth-regular graphs of odd girth.
\end{remark}

\end{document}